\documentclass[11pt,UTF-8,reqno]{amsart}
\usepackage{enumerate}
\usepackage{mhequ}
\usepackage[margin=1.0in]{geometry}
\usepackage{amssymb,url,color, booktabs,nccmath}
\usepackage{mathrsfs}
\usepackage{enumitem}
\usepackage{graphicx}
\usepackage{tikz}
\usepackage{graphicx}

\usepackage{color}
\usepackage[colorlinks=true]{hyperref}
\hypersetup{
    linkcolor=blue,          
    citecolor=red,        
    filecolor=blue,      
    urlcolor=cyan
}

\definecolor{darkergreen}{rgb}{0.0, 0.5, 0.0}

\numberwithin{equation}{section}

\newcommand{\be}{\begin{eqnarray}}
\newcommand{\ee}{\end{eqnarray}}
\newcommand{\ce}{\begin{eqnarray*}}
\newcommand{\de}{\end{eqnarray*}}
\newtheorem{theorem}{Theorem}[section]
\newtheorem{lemma}[theorem]{Lemma}
\newtheorem{remark}[theorem]{Remark}
\newtheorem{definition}[theorem]{Definition}
\newtheorem{proposition}[theorem]{Proposition}
\newtheorem{Examples}[theorem]{Example}
\newtheorem{corollary}[theorem]{Corollary}

\def\proj{\mathbf{p}}

\def\Re{{\mathrm{Re}}}

\def\u{\mathbf{u}}

\def\[{{\Big[}}
\def\]{{\Big]}}
\def\<{{\langle}}
\def\>{{\rangle}}
\def\({{\Big(}}
\def\){{\Big)}}

\def\bx{{\mathbf{x}}}
\def\tr{\mathrm {tr}}

\def\dif{{\mathord{{\rm d}}}}

\def\min{{\mathord{{\rm min}}}}

\def\no{\nonumber}
\def\={&\!\!=\!\!&}

 \newcommand{\eqdef}{\stackrel{\mbox{\tiny def}}{=}}

\def\cI{{\mathcal I}}

\def\cP{{\mathcal P}}
\def\cQ{{\mathcal Q}}

\def\cS{{\mathcal S}}

\def\mD{{\mathbb D}}
\def\mE{{\mathbb E}}

\def\mM{{\mathbb M}}
\def\mN{{\mathbb N}}

\def\mS{{\mathbb S}}
\def\mT{{\mathbb T}}

\def\mZ{{\mathbb Z}}

\def\1{{\mathbf{1}}}

\def\E{\mathbf E}

\def\geq{\geqslant}
\def\leq{\leqslant}

\def\le{\leqslant}

\newcommand{\fs}{\mathbb{S}}
\newcommand{\fd}{\mathbb{D}}
\newcommand{\ftw}{\mathbb{T}}
\newcommand{\fst}{\mathbb{M}}
\newcommand{\fex}{\mathbb{E}}

\def\E{\mathbf{E}}
\def\CP{\mathcal{P}}
\def\Tr{\mathrm{Tr}}
\def\q{\mathfrak{q}}

\def\Re{{\mathrm{Re}}}

\def\u{\mathbf{u}}

\def\[{{\Big[}}
\def\]{{\Big]}}
\def\<{{\langle}}
\def\>{{\rangle}}
\def\({{\Big(}}
\def\){{\Big)}}

\def\bx{{\mathbf{x}}}
\def\tr{\mathrm {tr}}

\def\dif{{\mathord{{\rm d}}}}

\def\min{{\mathord{{\rm min}}}}

\def\tr{{\rm Tr}}
\def\no{\nonumber}
\def\={&\!\!=\!\!&}
\def\bt{\begin{theorem}}
\def\et{\end{theorem}}
\def\bl{\begin{lemma}}
\def\el{\end{lemma}}
\def\br{\begin{remark}}
\def\er{\end{remark}}
\def\bx{\begin{Examples}}
\def\ex{\end{Examples}}
\def\bd{\begin{definition}}
\def\ed{\end{definition}}
\def\bp{\begin{proposition}}
\def\ep{\end{proposition}}
\def\bc{\begin{corollary}}
\def\ec{\end{corollary}}

\def\so{\mathfrak{so}}
\def\su{\mathfrak{su}}
\def\u{\mathfrak{u}}

\def\mfg{\mathfrak{g}}

\def\geq{\geqslant}
\def\leq{\leqslant}

\def\le{\leqslant}

\def\R{\mathbb R}
\def\C{\mathbb C}

\def\<{\langle} \def\>{\rangle}

\allowdisplaybreaks

\begin{document}


\date{\today}


\title{A new derivation of the finite $N$ master loop equation for lattice Yang-Mills}

\author{Hao Shen}
\address[H. Shen]{Department of Mathematics, University of Wisconsin - Madison, USA}
\email{pkushenhao@gmail.com}

\author{Scott A. Smith}
\address[S. A. Smith]{Academy of Mathematics and Systems Sciences, Chinese Academy of Sciences, Beijing, China
}
\email{ssmith@amss.ac.cn}

\author{Rongchan Zhu}
\address[R. Zhu]{Department of Mathematics, Beijing Institute of Technology, Beijing 100081, China 
}
\email{zhurongchan@126.com}

\maketitle

\begin{abstract}
We give a new derivation of the finite $N$ master loop equation for  lattice Yang-Mills theory with structure group $SO(N)$, $U(N)$ or $SU(N)$.
The $SO(N)$ case was initially proved by Chatterjee in \cite{Cha},
and $SU(N)$ was analyzed in a follow-up work by Jafarov \cite{Jafar}.
  Our approach is based on the Langevin dynamic, an SDE on the manifold of configurations, and yields a simple proof via It\^o's formula.
\end{abstract}

\section{Introduction}

Let  $G$ be the Lie group
$SO(N)$, $U(N)$, or $SU(N)$. The goal of this paper is to derive the finite $N$ master loop equation
for the lattice Yang-Mills theory with gauge group $G$. This was first obtained in  \cite{Cha} for $G=SO(N)$ and then in \cite{Jafar} for  $G=SU(N)$.

We recall the setup, closely following the notation in \cite{Cha}.
Let $\Lambda$ be a finite subset of $\mZ^d$.
We recall that a lattice edge is positively oriented if the beginning point is smaller in lexographic order than the ending point.
Let $E^+$ (resp. $E^-$) be the set of positively (resp. negatively) oriented edges,
and  denote by $E_\Lambda^+$, $E_\Lambda^-$ the corresponding subsets
of edges with both beginning and ending points in $\Lambda$. Define $E\eqdef E^+\cup E^-$ and let $u(e)$ and $v(e)$ denote the starting point and ending point of an edge $e\in E$, respectively. 
A path $\rho$ in the lattice $\mZ^d$ is defined to be a sequence of edges $e_1e_2\cdots e_n$ with $e_i\in E$ and $v(e_i)=u(e_{i+1})$ for $i=1,2,\cdots, n-1$. The path $\rho$ is called closed if $v(e_n)=u(e_1)$.  A plaquette is a closed path of length four which traces out the boundary of a square; more precisely it is non-backtracking in the sense of  \cite[Sec.~2]{Cha}.   The  set of plaquettes is denoted as $\cP$ and $\CP_\Lambda$ is the set of plaquettes whose vertices are all in $\Lambda$, and
$\CP^+_\Lambda$ is the subset of plaquettes $p=e_1e_2e_3e_4$ such that
the beginning point of $e_1$ is lexicographically the smallest among all the vertices in $p$ and the ending point of $e_1$ is the second smallest.

The lattice Yang-Mills theory (or lattice gauge theory)
on $\Lambda$ for the gauge group $G$, with $\beta\in\R$ the inverse coupling constant, is the
  probability measure $\mu_{\Lambda, N, \beta}$  on the set of all collections $Q = (Q_e)_{e\in E_\Lambda^+}$ of $G$-matrices, defined as
\begin{equation}\label{measure}
\dif\mu_{\Lambda, N, \beta}(Q)
 := Z_{\Lambda, N,\beta}^{-1}
\exp\biggl(N\beta \, \Re \sum_{p\in \CP^+_\Lambda} \Tr(Q_p)\biggr) \prod_{e\in E^+_\Lambda} \dif\sigma_N(Q_e)\, ,
\end{equation}
where $Z_{\Lambda,  N, \beta}$ is the normalizing constant, $Q_p \eqdef Q_{e_1}Q_{e_2}Q_{e_3}Q_{e_4}$ for a plaquette $p=e_1e_2e_3e_4$, and $\sigma_N$ is the Haar measure on $G$.
Note that for $p\in \CP^+_\Lambda$ the edges $e_{3}$ and $e_{4}$ are negatively oriented,
so throughout the paper we define $Q_{e}\eqdef Q_{e^{-1}}^{-1}$ for $e \in E^{-}$, where $e^{-1}$ denotes the edge with orientation reversed.
Also, $\Re$ is the real part, which can be omitted when $G=SO(N)$.  We do not intend to further discuss the background and motivation for the above model \eqref{measure}; instead we refer to the review paper \cite{Chatterjee18}.

 For a closed path $\rho=e_1\cdots e_n$, $\rho'$ is said to be cyclically equivalent to $\rho$ if $\rho'=e_ie_{i+1}\cdots e_ne_1e_2\cdots e_{i-1}$ for some $2\leq i\leq n$.  Cyclical equivalence classes are referred to as cycles, and a cycle with no backtracking is called a loop and denoted by $l$. By \cite[Lemma 2.1]{Cha}, for any cycle $l$  there is a unique loop denoted as $[l]$ by successive backtrack erasures until there are no more backtracks. A loop sequence  $s=(l_{1},\dots,l_{m})$ is a collection of loops; more precisely it is an equivalence class understood up to a insertion and deletion of a null cycle (think of $ee^{-1}$ for instance).  The length of a loop $l$ is denoted by $|l|$.  For a loop sequence $s$ with minimal representation $(l_1,\dots, l_n)$, the length is defined as
 \begin{equation}
 |s|\eqdef \sum_{i=1}^n|l_i| . \label{eq:DefS}
 \end{equation}
 We refer to  \cite[Sec.~2]{Cha} for precise definitions of  loop sequence, minimal representation,  location, etc.

Given a loop $ l  = e_1 e_2 \cdots e_n$, the Wilson loop variable $W_l$ is defined as
$$
W_l = \Tr (Q_{e_1}Q_{e_2}\cdots Q_{e_n})\;.
$$
By cyclic invariance of the trace, this definition is independent of the particular representative chosen in the equivalence class $l$. Write $\E$  
for expectation with respect to \eqref{measure}. For any non-null loop sequence $s$ with minimal representation $(l_1,\ldots,l_m)$ such that each $l_i$ is contained in $\Lambda$, define
\begin{equ}[e:phi]
W_s\eqdef W_{l_1}W_{l_2}\cdots W_{l_m},\quad \phi(s) \;\eqdef \;\E \frac{W_s}{N^m}\;.
\end{equ}
 The master loop equation is a recursion which expresses $\phi(s)$ in terms of a linear combination of $\phi(s')$, where $s'$ is a loop sequence obtained by performing an operation on $s$.  The operations are called splitting, twisting, merger, deformation, and expansion; each being further divided into a positive or negative type. This leads to a definition of sets $\ftw^{\pm}(s)$, $\fs^{\pm}(s)$, $\fst^{\pm}(s)$, $\fd^{\pm}(s)$, $\mE^{\pm}(s)$ of all loop sequences obtained by performing one of these operations on $s$.  The precise definition is given in \ref{split}-\ref{expan} in Section \ref{sec:pf}, but we also recommend the graphs and further discussion in \cite[Sec.~2.2]{Cha} for additional intuition.   

Moreover, to state the theorem, we define as in \cite{Jafar}
\begin{equ}[e:ell]
\ell(s) \eqdef \sum_{e\in E^+}t(e)^2\;,
\qquad
t(e)\eqdef \sum_{i=1}^m t_i(e)\quad  (e\in E^+)
\end{equ}
and $t_i(e) \eqdef $ the number of occurrences of $e$ minus the number of occurrences of $e^{-1}$ 
in the loop $l_i$. 
 The following is the finite $N$ master loop equation for the model \eqref{measure}. 
\begin{theorem}\label{theo:main}
Let $s$ be as above, and suppose that all vertices that are at distance $\le 1$ from any $l_i$ belong to $\Lambda$.
Then for $G=SO(N)$ (\cite[Theorem~3.6]{Cha})
\begin{equs}
(N -1)|s|\phi(s) & =N\beta \!\!\!\sum_{s'\in \fd^-(s)}\!\!\phi(s')
 - N\beta \!\!\! \sum_{s'\in \fd^+(s)}\!\!\phi(s')
+ N \!\!\! \sum_{s'\in \fs^-(s)} \!\! \phi(s')
 - N \!\!\! \sum_{s'\in \fs^+(s)}\!\! \phi(s')
 \\
&\quad 
 +\sum_{s'\in \ftw^-(s)} \!\! \phi(s')
 - \sum_{s'\in \ftw^+(s)} \!\! \phi(s')+ \frac{1}{N}\!\!\! \sum_{s'\in \fst^-(s)}\!\!\phi(s')
- \frac{1}{N}\!\!\! \sum_{s'\in \fst^+(s)}\!\!\phi(s')\;.	\label{e:master}
\end{equs}
For $G=SU(N)$ (\cite[Theorem~4.2]{Jafar}) 
\begin{equs}[e:jafar]
\Big(N |s|-\frac{\ell(s)}{N}\Big)\phi(s)
&=\frac{N\beta}{2} \!\!\!  \sum_{s'\in \fd^-(s)} \!\!\phi(s')
-\frac{N\beta}{2}\!\!\!  \sum_{s'\in \fd^+(s)}\!\!\phi(s')
+N \!\!\!  \sum_{s'\in \fs^-(s)}\!\!\phi(s')
-N \!\!\!  \sum_{s'\in \fs^+(s)}\!\!\phi(s')
\\
&\;
+\frac{N\beta}{2}\!\!\!  \sum_{s'\in \fex^-(s)}\!\!\phi(s')
-\frac{N\beta}{2}\!\!\! \sum_{s'\in \fex^+(s)}\!\!\phi(s')
+\frac{1}{N} \!\!\!  \sum_{s'\in \fst_{U}^-(s)}\!\!\phi(s')
-\frac{1}{N}\!\!\!  \sum_{s'\in \fst_{U}^+(s)}\!\!\phi(s')\, .
\end{equs}
\footnote{Compared to \cite[Theorem~4.2]{Jafar} we use a different notation $\mM_{U}^{\pm}$ for the sets of merger terms to distinguish it from $SO(N)$ case.}
For $G=U(N)$,
\begin{equs}[eq:phi1UN]
	N|s|\phi(s) & =\frac{N\beta}2 \!\!\! \sum_{s'\in \mathbb{D}^-(s)}\!\!\phi(s')
	-\frac{N\beta}{2} \!\!\! \sum_{s'\in \mathbb{D}^+(s)} \!\!\phi(s')
	+N\!\!\!\sum_{s'\in \mS^{-}(s)}\!\!\phi(s')
	-N\!\!\! \sum_{s'\in \mS^{+}(s)}\!\!\phi(s')
\\&\qquad\qquad\qquad\qquad
+\frac1N \!\!\!\sum_{s'\in \mM_{U}^{-}(s)}\!\!\phi(s')
-\frac1N\!\!\! \sum_{s'\in \mM_{U}^{+}(s)}\!\!\phi(s').
\end{equs}
\end{theorem}

The proof of this theorem in \cite{Cha,Jafar} is based on Stein's exchangeable pairs and integration by parts (see Sec.~5 - Sec.~8 therein). Here we reprove it using a simple Langevin dynamic and It\^o calculus. The Langevin dynamic is given in \eqref{eq:YM} below. 
To illustrate our basic idea of using a dynamic in place of Stein's method or integration by parts, recall that  for a standard Gaussian $X$, one has integration by parts (Stein's lemma)
$\E [Xf(X)] = \E [f'(X)]$.
This identity can be derived in a simple way from It\^o's formula applied to a Langevin dynamic as follows. Consider the Ornstein-Uhlenbeck process
$\dif X_t = - \frac12 X_t \dif t + \dif B_t$ for a Brownian motion $B$.
Let $F$ be an anti-derivative of $f$, i.e. $F' = f$. Then
$\dif F(X_t) = F'(X_t) (-\frac12 X_t \dif t + \dif B_t) + \frac12 F''(X_t) \dif t$.
In stationarity, taking expectation yields Stein's lemma.
Our proof of Theorem~\ref{theo:main} is based on a generalization of this idea.

As applications of the master loop equations, 
Chatterjee and Jafarov \cite{Cha,Jafar} proved 
various properties of Wilson loops in the large $N$ limit under a smallness constraint in $\beta$, such as 
a discrete surface sum formula (gauge-string duality),
the factorization property of Wilson loops,  an area law upper bound, and real analyticity in $\beta$.
Further investigations were initiated by Basu and Ganguly \cite{MR3861073}, where more structure is deduced on the limit in the two dimensional lattice setting. 
In this paper we only focus on our new derivation of the loop equations so we do not restate these results.
However let us 
mention that the Langevin dynamic also has many other applications besides deriving loop equations.
In  \cite{SZZ22}, it is shown that  the Langevin dynamic
 can also be used to prove large $N$ limit, factorization, mass gap property and  uniqueness of the lattice Yang--Mills model on the infinite-volume lattice, again under smallness condition in $\beta$ (with the bound on $\beta$ therein is explicit and uniform in $N$).
We also remark that dynamics very similar with ours 
seem to have appeared in physics and are used for Monte Carlo simulations for lattice gauge theory, see e.g. \cite{batrouni1985,guha1983}.

Master loop equations for lattice gauge theories were originally stated in physics literature, see
Makeenko--Migdal \cite{MM1979}, Foerster \cite{Foerster1979}, Eguchi \cite{Eguchi1979}.
For more recent physics literature, see \cite[(2.18)]{AndersonKruczenski} for the $SU(N)$ case.
\footnote{Besides more precise notation and rigorous proof as done in \cite{Cha,Jafar} and this paper, we point out that a factorization property of Wilson loops when passing from finite $N$ loop equations to $N\to \infty$ equations
was often assumed in physics literature, whereas \cite{Cha,Jafar,SZZ22} could rigorously prove it for small $\beta$.} 
The first rigorous version was established for two-dimensional Yang-Mills model in continuum in \cite{Levy11}, 
and alternative proofs and extensions were given in \cite{MR3554890,MR3982691,MR3631396,MR3613519} on plane or surfaces. 
These equations belong to a general class of equations arising in quantum field theory and random matrix theory, known as Integration by Parts or Schwinger--Dyson equations. 
See for instance  \cite{CGM2009,GuionnetNovak} who derived  Schwinger--Dyson equations 
for orthogonal and unitary multimatrix models, which are to some extent related with the lattice Yang--Mills model. 

We remark that the lattice cutoff is a key simplification which allows us to prove Theorem~\ref{theo:main} in any dimension; it would be much more challenging to prove similar results in the continuum (besides the aforementioned two-dimensional results).
It would be natural to conjecture that the Langevin dynamic in Section~\ref{sec:SDE}
- which is the main ingredient in our proof - has 
a scaling limit given by the ones recently constructed in \cite{CCHS2d,CCHS3d} on the two and three dimensional torus. 
\footnote{This scaling limit result on   two dimensional torus  was obtained in \cite{Chevyrev2023} after completion of this paper.}
We believe that our method via Langevin dynamic is robust enough 
to derive master loop equations for more complicated models such as Yang--Mills--Higgs model on the lattice (or even more general gauge theories coupled with matter fields), which will be considered in our future work. 
We also refer to 
\cite[Sec.~18]{Cha} for a more comprehensive list of open questions; for instance, the possibility of rigorously deriving an Eguchi--Kawai type reduction formula in large $N$, see the physics references \cite{EguchiKawai1982,Gonzalez1983,Gonzalez2014} or the book \cite[Part~4]{MakeenkoBook}.

\subsection*{Acknowledgments}
H.S. gratefully acknowledges financial support from NSF grants DMS-1954091 and CAREER DMS-2044415, and helpful discussions with Ilya Chevyrev.  S.S. and R.Z. are grateful to the financial supports by National Key R\&D Program of China (No. 2022YFA1006300).
R.Z. is grateful to the financial supports of the NSFC (No. 12271030), and BIT Science and Technology Innovation Program Project 2022CX01001 and the financial supports  by the Deutsche Forschungsgemeinschaft (DFG, German Research Foundation) – Project-ID 317210226--SFB 1283, and helpful discussions with Xin Chen.  S. A. Smith is grateful for financial support from the Chinese Academy of Sciences. 
We also would like to thank Todd Kemp for discussion on relevant references, and an anonymous referee
who drew our attention to very interesting and relevant papers
\cite{Foerster1979,Eguchi1979,AndersonKruczenski} and
 \cite{EguchiKawai1982,Gonzalez1983,Gonzalez2014}.

\section{Preliminaries}
\label{sec:Pre}

This section collects some standard facts about Brownian motions on a Lie group $G$ or its Lie algebra $\mfg$,
mostly from \cite[Sec.~1]{Levy11}. 

We write the Lie algebras of $SO(N)$, $U(N)$, $SU(N)$
as $\so(N)$, $\u(N)$, $\su(N)$ respectively.
Every matrix $Q$ in one of these Lie groups satisfies $QQ^* = I_N$, and every matrix $X$ in one of these Lie algebras
satisfies $X + X^* = 0$.
Here $I_N$ denotes the identity matrix, and
for any matrix $M$ we write $M^*$ for the conjugate transpose of $M$.
Let $M_N(\R)$ and $M_N(\C)$ be the space of real and complex $N\times N$ matrices.

We endow $M_N(\C)$ with the Hilbert-Schmidt inner product
\begin{equ}[e:HS]
	\< X,Y\> = \Re \Tr (X Y^*)
	\qquad  \forall X,Y\in M_N(\C).
\end{equ}
We restrict this inner product to our
Lie algebra $\mfg$, which is then invariant under the adjoint action.
In particular for $X,Y\in \so(N),\u(N)$ or $\su(N)$ we have
$\< X,Y\> =- \Tr (XY)$.

It is well-known that this induces an inner product
on the tangent space at every $Q\in G$
via the right multiplication on $G$.
Indeed, given any $X \in \mfg$, the curve $t \mapsto e^{tX} Q$ is well approximated near $t=0$ by $Q+tXQ$ up to an error of order $t^{2}$.
Hence, for $X,Y\in \mfg$, $XQ$ and $YQ$ are two tangent vectors on the tangent space at
$Q\in G$, and their inner product is given by $\Tr((XQ)(YQ)^*) = \Tr(XY^*)$.

Denote by $\mathfrak{B}$ and $B$ the Brownian motions on $G$ and its Lie algebra $\mfg$ respectively.
$B$ is a continuous Gaussian process characterized by
\begin{equation}\label{eq:DefBM}
\E \Big[  \<B(s),X \> \<B(t),Y \>  \Big] = \min(s,t) \< X,Y \>
\qquad
\forall
X,Y \in \mfg.
\end{equation}
By  \cite[Sec.~1.4]{Levy11}, the Brownian motions $\mathfrak{B}$ and $B$ are related through the following SDE:
\begin{equ}[e:dB]
	\dif \mathfrak B(t) = \dif B(t) \circ \mathfrak B(t) = \dif B(t)\, \mathfrak B(t)
	+ \frac{c_\mfg}{2} \mathfrak B(t) \dif t,
\end{equ}
where $\circ$ is the Stratonovich product, and $\dif B\, \mathfrak B$ is in the It\^o sense.
Here the constant $c_\mfg$ is determined by
$ \sum_{\alpha} v_\alpha^2  =c_\mfg I_N$ where
$(v_\alpha)_{\alpha=1}^{\dim_{\R}\mfg}$  is an orthonormal basis of $\mfg$. Note that $\dim_{\R}\mfg$ indicates that for matrices with complex entries, dimension is counted with respect to $\R$.
Moreover, by \cite[Lem.~1.2]{Levy11},
\begin{equ}[e:c_g]
	c_{\so(N)} =  -\frac12(N-1),
	\quad
	c_{\u(N)} =  -N,
	\quad
	c_{\su(N)} =  -\frac{N^2-1}{N} .
\end{equ}
Note that in \cite[Lem.~1.2]{Levy11}, the scalar product differs from \eqref{e:HS} by a scalar multiple depending on $N$ and $\mfg$, so we accounted for this in the expression for $c_\mfg$ above.

Denote by $\delta$  the Kronecker function, i.e. $\delta_{ij}=1$ if $i=j$ and $0$ otherwise. For any matrix $M$, we write $M^{ij}$ for its $(i,j)$th entry.
The following holds by straightforward calculations:
\minilab{e:BB}
\begin{equs}[2]
	\dif\< B^{ij}, B^{k\ell}\>
	&=\frac12(\delta_{ik}\delta_{j\ell}-\delta_{i\ell}\delta_{jk})\dif t,
	&\qquad \mfg=\so(N);		\label{e:BB1}
\\
	\dif\< B^{ij}, B^{k\ell}\>
	&= -\delta_{i\ell} \delta_{jk}\, \dif t\;,
	&\qquad\mfg=\u(N);		\label{e:BB2}
\\
	\dif\< B^{ij}, B^{k\ell}\>
	&= \big( -\delta_{i\ell} \delta_{jk} + \frac{1}{N} \delta_{ij}\delta_{k\ell} \big)\, \dif t\;,
	&\qquad\mfg=\su(N).				\label{e:BB3}
\end{equs}
\begin{remark}
The relation \eqref{e:BB3} can be deduced from \eqref{e:BB2} as follows.  Given a $\u(N)$ Brownian motion $t \mapsto B(t)$, we may define an $\su(N)$ Brownian motion $t \mapsto \hat{B}(t)$ by letting $\hat{B}(t)\eqdef B(t)-\frac{1}{N}\text{Tr}(B(t))I_N$.  Since any $X \in \su(N)\subset \u(N)$ is traceless, the identity \eqref{eq:DefBM} satisfied by $B$ implies the same identity for $\hat{B}$ in light of the equality $\langle \hat{B}(t) , X  \rangle=\langle B(t) , X \rangle$, which follows from $\langle I_N,X \rangle=\text{Tr}(X)=0$.  To obtain \eqref{e:BB3} from \eqref{e:BB2}, note that off diagonal entries of $\hat{B}$ are identical to those of $B$, so \eqref{e:BB3} holds if $i \neq j$ and $k \neq \ell$.  Similarly, since on-diagonal entries of $B$ are independent from off-diagonal entries, it holds that $\langle \hat{B}^{ii}, \hat{B}^{k\ell} \rangle=0$ for $k \neq \ell$, again consistent with \eqref{e:BB3}.  For $k=\ell$, note that 
\begin{equation}
\dif\langle \hat{B}^{ii},\hat{B}^{\ell \ell} \rangle= \dif \Big \langle B^{ii}-\frac{1}{N}\sum_{j=1}^{N}B^{jj},  B^{\ell\ell}-\frac{1}{N}\sum_{j=1}^{N}B^{jj} \Big \rangle= \big (-\delta_{i\ell}+\frac{1}{N}+\frac{1}{N}-N \cdot \frac{1}{N^{2}} \big ) \dif t=\big ( -\delta_{i\ell}+\frac{1}{N} \big ) \dif t, \nonumber
\end{equation}
where we used the independence of diagonal entries of $B$.
\end{remark}

\begin{remark}
	Note that the choice of this inner product \eqref{e:HS} may differ among the literature
	by a  constant multiple. \eqref{e:c_g} will then differ by (the inverse of) the same constant. The r.h.s. of \eqref{e:BB1} -- \eqref{e:BB3} should then also be multiplied by the suitable constants. 
\end{remark}

\section{Yang Mills SDE}
\label{sec:SDE}

Define the configuration space as the Lie group product  $\cQ=G^{E_{\Lambda}^+}$,
consisting of all maps $Q:e \in E_{\Lambda}^{+} \mapsto Q_{e} \in G$.
Let $\q=\mfg^{E_{\Lambda}^+}$ be the corresponding direct sum of $\mfg$ and note that $\q$ is the Lie algebra of the Lie group $\cQ$.
For any matrix-valued functions $A,B$ on $E_{\Lambda}^+$, we denote by $AB$
the pointwise product.
 Given $X \in \q$, the exponential map $t \mapsto \text{exp}(tX)  $ is defined by $\text{exp}(tX)_{e}\eqdef e^{tX_{e}}$.

 As above, the tangent space at $Q \in \mathcal{Q}$ consists of the products $XQ$ with $X \in \q$, and given two such elements $XQ$ and $YQ$, their inner product is defined by
\begin{equation}
\langle XQ,YQ \rangle = \sum_{e \in E_{\Lambda}^{+} } \text{Tr}(X_{e}Y_{e}^{*}) \nonumber.
\end{equation}
Given any function $f \in C^{\infty}(\mathcal{Q})$, the right-invariant derivative is given by $\frac{\dif}{\dif t}|_{t=0} f(\text{exp}(tX)  Q)$.  For each $Q \in \mathcal{Q}$, the differential $\nabla f(Q)$ is an element of the tangent space at $Q$ which satisfies for each $X \in \q$
\begin{equation}
\langle \nabla f(Q),  XQ \rangle = \frac{\dif}{\dif t}\Big|_{t=0} f(\text{exp}(tX)  Q). \label{eee1}
\end{equation}

 Denote by $\mathfrak{B} = (\mathfrak{B}_e)_{e\in E_\Lambda^+}$
and $B = (B_e)_{e\in E_\Lambda^+}$
the Brownian motions on $\cQ$ and $\q$ respectively, where $\mathfrak{B}_e$ and $B_e$ are related through \eqref{e:dB} for each $e\in E_\Lambda^+$. For any two edges $e_{1}, e_{2} \in E_\Lambda^+$ with $e_{1} \neq e_{2}$, the pairs $(\mathfrak{B}_{e_{1}},B_{e_{1}})$ and $(\mathfrak{B}_{e_{2}},B_{e_{2}})$ are independent. 
Let $\mathcal S (Q) \eqdef N\beta \Re \sum_{p\in \CP^+_\Lambda} \Tr(Q_p)$.
We consider the Langevin dynamic for the measure \eqref{measure}, which is
	the following SDE on $\cQ$
\begin{equ}[e:langevin]
	\dif Q =\frac12 \nabla \mathcal S (Q) \dif t + \dif \mathfrak B\;.
\end{equ}

We now derive an explicit expression for $ \nabla \mathcal S$.
For a plaquette $p=e_1e_2e_3e_4 \in \CP$, we write $p\succ e_1$ to indicate that $p$ is a plaquette that starts from edge $e_{1}$.
Note that for each edge $e$, there are $2(d-1)$ plaquettes  in $\CP$
such that $p\succ e$. 
For {\it any} Lie algebra $\mfg$ embedded into $M_N(\C)$, it forms a closed subspace of $M_N$,
 and therefore $M_N$ has an orthogonal decomposition $M_N=\mfg \oplus \mfg^\perp$.
Given $M\in M_N$, we denote by $\proj M \in \mfg$ the orthogonal projection onto $\mfg$.
\begin{lemma}
It holds that
\begin{equ}[e:grad]
\nabla \cS(Q)_{e}=  N\beta \sum_{p\in \cP_\Lambda, p\succ e} \proj Q_{p}^{*} \cdot (Q_{e}^{*})^{-1}\;,
\end{equ}
where $\cdot$ is matrix multiplication.
\end{lemma}
\begin{proof}
To prove the claim, fixing an edge $e\in E^+_\Lambda$,
let $X \in \mfg$ and with a slight abuse of notation we write
$X\in \q$ for the function which equals  $X$ at $e$ and zero elsewhere.
Note that for
	every $\tilde p\in \CP^+_\Lambda$ that contains both the beginning point and the ending point of $e$,
there is a unique way to obtain  a plaquette
$p\in \CP_\Lambda$ such that $p\succ e$ by a cyclic permutation  and possibly a  reversal of the four edges.
One then has  $\Tr(Q_{\tilde p})=\Tr(Q_p)$ (without reversal)
or $\Tr(Q_{\tilde p})=\Tr(Q_p^{-1})=\Tr(Q_p^*)$ (with reversal).
We have
$$
\frac{\dif}{\dif t} \Big|_{t=0}
\Re \Tr(\text{exp} (tX) Q_p)
=
\frac{\dif}{\dif t} \Big|_{t=0}
\Re \Tr( (\text{exp} (tX) Q_p)^*)
=\text{Re}\text{Tr}(XQ_{p})
$$
where we used $(XQ_p)^* = \overline{(XQ_p)^t}$.
Therefore the
derivative of $\cS$ at $Q$ in the tangent direction $XQ$
is equal to $N\beta$ times
\begin{equ}
\sum_{p\in \cP_\Lambda,p\succ e } \text{Re}\text{Tr}(XQ_{p}) \nonumber
=\sum_{p\in \cP_\Lambda,p\succ e }
\langle X, Q_{p}^{*} \rangle =\sum_{p\in \cP_\Lambda,p\succ e } \langle X,  \proj Q_{p}^{*} \rangle
=\sum_{p\in \cP_\Lambda,p\succ e } \text{Re} \text{Tr}(X^{*}\proj Q_{p}^{*} )  .\label{gen1}
\end{equ}
Furthermore, note that
\begin{equation}
\langle \nabla \cS(Q)_{e} ,  XQ_{e} \rangle=\text{Re}\text{Tr}(\nabla \cS(Q)_{e} Q_{e}^{*}X^{*})=\text{Re}\text{Tr}( X^{*} \nabla \cS(Q)_{e}Q_{e}^{*} )\;. \label{gen2}
\end{equation}
Keeping in mind \eqref{eee1}, to ensure that \eqref{gen1} and \eqref{gen2} agree,
we have \eqref{e:grad}.
\end{proof}

For our specific choice of Lie algebras, our SDE system reads
\begin{equ}[eq:YM]
\dif Q_e =\frac12 \nabla \mathcal S (Q)_e \dif t
+\frac12 c_{\mathfrak g}Q_e\dif t+\dif B_eQ_e \;,    \qquad  (e\in E_\Lambda^+)
\end{equ}
where $c_{\mathfrak g}$ is as in \eqref{e:c_g} and
\begin{equ}[e:DS]
\frac12 \nabla \cS(Q)_{e}=
\begin{cases}
 \displaystyle 
 -\frac14N\beta \sum_{p\in \cP_\Lambda,p\succ e}(Q_p-Q_p^*)Q_e\;,
&\qquad
\mfg \in  \{\so(N),\u(N)\}\;,
\\
 \displaystyle - \frac14 N\beta\sum_{p\in \cP_\Lambda ,p\succ e}
	\Big( (Q_p-{Q}_p^{*}) - \frac{1}{N}\tr(Q_p-{Q}_p^{*}) I_N\Big)   Q_e\;,
&\qquad
\mfg \in \su(N)\;.
\end{cases}
\end{equ}

While our measure \eqref{measure} and the dynamic are both defined on the
configuration space $\cQ=G^{E_{\Lambda}^+}$ (in particular the above SDE system
is parametrized by $e\in E_\Lambda^+$), when we apply it to Wilson loops later,
we will also need to consider the dynamic of $Q_{e^{-1}}$ for  $e\in E_\Lambda^+$,
which is just $Q_{e}^*$. So we give the conjugate transpose of \eqref{eq:YM}:
\begin{equ}[eq:YM*]
\dif Q_e^* =\frac12 \big(\nabla \mathcal S (Q)_e\big)^* \dif t
+\frac12 c_{\mathfrak g}Q_e^*\dif t+Q_e^*\dif B_e^* \;,
	\qquad  (e\in E_\Lambda^+)\;.
\end{equ}

This system is well-posed and 
has \eqref{measure} as invariant measure, 
as we show in the next two lemmas.

\bl\label{lem:exist} For fixed $N\in\mN$, $T>0$ and any initial data $Q(0)=(Q_e(0))_{e\in E_\Lambda^+}\in \cQ$, there exists a  unique solution $Q=(Q_e)_{e\in E_\Lambda^+}\in C([0,T];\cQ)$ to \eqref{eq:YM} a.s..
\el

\begin{proof} For fixed $N$ and $\Lambda$, we can write \eqref{eq:YM} as the system for the entries of the matrices $Q_e$, which can be viewed as a finite dimensional SDE with locally Lipschitz coefficients. We introduce a stopping time
	$$\tau:=\inf \{t\geq 0: \|Q_e(t)\|_{\infty}>2, \, \, \text{for at least one }  \, e\in E^+_\Lambda\}\wedge T,$$
 and we obtain  local solutions $Q=(Q_e)_{e\in E_\Lambda^+}$ with $Q_e\in C([0,\tau];M_N)$ for $e\in E_\Lambda^+$,  which satisfies \eqref{eq:YM} before $\tau$, by fixed point argument (see e.g. \cite[Chapter~3]{MR3410409}).

Since for  $e\in E_\Lambda^+$, 
$\nabla \cS(Q)_e$ belongs to the tangent space of $G$ at $Q_e$, 
 exactly the same argument as in \cite[Lemma 1.3]{Levy11} imply that $Q_e(t)\in G$, $\forall t\geq0$,  and $\tau=T$ a.s.. 
 Indeed, by \cite[Lemma 1.3]{Levy11}, the Brownian motion  $\mathfrak{B}_e$ in $G$ satisfies $\mathfrak{B}_e(t)\in G$ for every $t\geq0$ a.s.. 
 By standard calculation (cf. \cite[Chapter~3]{Hsu}) the generator $L$ for our SDE with $F\in C^\infty(\cQ)$  is given by
 	\begin{align}\label{e:LF}
 		LF=	\frac12\sum_{e\in E_\Lambda^+}\Delta_{Q_e}F+\sum_{e\in E_\Lambda^+}\frac12\<\nabla \cS(Q)_e,\nabla_{Q_e} F\>.
 	\end{align}
 	Here $\Delta_{Q_e}$ and $\nabla_{Q_e}$ are the
 	Laplace--Beltrami operator and the gradient  (w.r.t. the variable $Q_e$) on $G$ endowed with
 	the metric given in Sec~\ref{sec:Pre}.
 When calculating $\dif (Q_eQ_e^*)$ using It\^o's formula, the first sum in $L(Q_eQ_e^*)$ from \eqref{e:LF}  vanishes as in the calculation in  \cite[Lemma 1.3]{Levy11}, and the second sum from $L(Q_eQ_e^*)$ is also zero since $\nabla \cS(Q)_e$ belongs to the tangent space of $G$ at $Q_e$.
\end{proof}

\begin{lemma}\label{lem:inv}
	\eqref{measure} is invariant under the SDE system \eqref{eq:YM}.
\end{lemma}
\begin{proof}
 Recall the generator in \eqref{e:LF}. Using integration by parts w.r.t. the Haar measure, we have for $F, G\in C^\infty(\cQ)$
\begin{align}\label{eq:in}
	\int (LF) G\dif \mu_{\Lambda,N,\beta}=&-\frac12\sum_{e\in E_\Lambda^+}\int\<\nabla_{Q_e}F,\nabla \cS(Q)_e\>G\dif \mu_{\Lambda,N,\beta}+\frac12\sum_{e\in E_\Lambda^+}\int\<\nabla_{Q_e}F,\nabla \cS(Q)_e\>G\dif \mu_{\Lambda,N,\beta}\no
	\\&-\frac12\sum_{e\in E_\Lambda^+}\int\<\nabla_{Q_e}F,\nabla_{Q_e} G \>\dif \mu_{\Lambda,N,\beta}\\=&-\frac12\sum_{e\in E_\Lambda^+}\int\<\nabla_{Q_e}F,\nabla_{Q_e} G \>\dif \mu_{\Lambda,N,\beta}=	\int (LG) F\dif \mu_{\Lambda,N,\beta}\no,
\end{align}
where we exchange $F$ and $G$ in the last step. Hence, $L$ is symmetric w.r.t. $\mu_{\Lambda,N,\beta}$ and the result follows by $L1\equiv0$. 
\end{proof}

Using Lemma \ref{lem:exist}, we can choose $\mu_{\Lambda, N,\beta}$ as an initial distribution and obtain a stationary solution $Q\in C([0,T];\cQ)$. We will fix this stationary solution in the following proof.


\section{Proof of the main theorem}\label{sec:pf}

The proof of Theorem~\ref{theo:main} is based on the It\^o formula applied to $W_{s}$ according to the dynamics \eqref{eq:YM}. We consider initial datum distributed according to $\mu_{\Lambda, N, \beta}$, so that the solution is stationary in law, c.f. Lemma \ref{lem:inv}.  The master loop equation for $\phi(s)$ is obtained by normalizing and taking expected value. The nonlinear drift  of the SDE yields the deformation and expansion terms,
 while the It\^o correction leads to the splitting, twisting and merger terms, in addition to a multiple of $W_{s}$, which combined with $c_\mfg$ part gives the left-hand side of the master loop equation
\eqref{e:master}.

We recall from   \cite[Sec.~2]{Cha} and \cite[Sec.~2]{Jafar}
the following notation and definitions.

\begin{enumerate}[leftmargin=*,label=(O\arabic*)]\setlength\itemsep{0.5em}
	\item		\label{split}
	$\times^1_{x,y} l$ and $\times^2_{x,y} l$ denote the
	pair of loops obtained by
	positive splitting of $l$ at $x$ and $y$
	if $l$ contains the same edge $e$ at $x$ and $y$,
	or
	negative splitting
	if $l$ contains $e$ at location $x$ and $e^{-1}$ at location $y$.
	For $l=aebec$ (where $a, b, c$ are paths and $e$ is an edge),  	$\times^1_{x,y} l \eqdef [aec]$ and $\times^2_{x,y} l \eqdef [be]$. For $l=aebe^{-1}c$, 	$\times^1_{x,y} l \eqdef [ac]$ and $\times^2_{x,y} l \eqdef [b]$.
	We say that a loop sequence $s'$ is obtained from splitting $s$ provided that exactly two components of $s'$ arise from splitting a single loop in $s$.  
	See  \cite[Fig.~8, Fig.~9]{Cha} and \cite[Fig.~6, Fig.~7]{Jafar} for graphical illustrations of positive and negative splittings.
	
	The sets $\mS^{+}(s)$ and $\mS^{-}(s)$ consist of all loop sequences obtained from positive or negative splitting of $s$, respectively.  
	\item		\label{twist}
	$\propto_{x,y} l $ denotes the negative twisting
	if $l$ contains an edge $e$ at both $x$ and $y$,
	or positive twisting if $l$ contains an edge $e$ at location $x$ and $e^{-1}$ at location $y$. For $l=aebec$, $\propto_{x,y} l \eqdef [ab^{-1}c]$. For  $l=aebe^{-1}c$, $\propto_{x,y} l \eqdef [aeb^{-1}e^{-1}c]$.  We say that a loop sequence $s'$ is obtained from twisting $s$ provided that exactly one component of $s'$ arises from twisting one loop in $s$. 
	See  \cite[Fig.~10, Fig.~11]{Cha} for graphical illustrations of positive and negative twistings.
	
	The sets $\mT^{+}(s)$ and $\mT^{-}(s)$ consist of all loop sequences obtained from positive or negative twisting of $s$, respectively.
	\item		\label{merger}
	$l\oplus_{x,y} l' $ and
	$l\ominus_{x,y} l'$  are positive and negative mergers of $l$ and $l'$
	at locations $x,y$. For $l=aeb$ and $l'=ced$ (where $a,b,c,d$ are paths and $e$ is an edge), $l\oplus_{x,y} l' =[aedceb]$, $l\ominus_{x,y} l'=[ac^{-1}d^{-1}b]$. For $l=aeb$ and $l'=ce^{-1}d$, $l\oplus_{x,y} l' =[aec^{-1}d^{-1}eb]$, $l\ominus_{x,y} l'=[adcb]$ (here $x$ and $y$ are the unique location in $l$ and $l'$, respectively, where $e$ or $e^{-1}$ occurs and $e$ is the edge occurring at location $x$ in $l$).  We say that a loop sequence $s'$ is obtained from merging $s$ provided that exactly one component of $s'$ arises from merging two loops in $s$.  See  \cite[Fig.~4, Fig.~5]{Cha} and \cite[Fig.~2, Fig.~3]{Jafar} for graphical illustrations of positive and negative mergers.
	
	The sets $\mM^{+}(s)$ and $\mM^{-}(s)$ denote all loop sequences obtained from either positive mergers or negative mergers of $s$.  Furthermore, we define two more sets $\mM^{+}_{U}(s) \subset \mM^{+}(s)$ and $\mM^{-}_{U}(s) \subset \mM^{-}(s)$; the first consists of positive mergers resulting from an edge $e$ appearing in both of the two merged loops; the second consists of negative mergers where an edge $e$ occurs in one loop and $e^{-1}$ in the other. 
	\item  \label{deform}
	$l \oplus_x p$ and $l\ominus_x p$
	are deformations obtained
	by merging $l$ and $p$ at locations $x$ and $y$
	(here $y$ is the unique location in $p$ where $e$ or $e^{-1}$ occurs and $e$ is the edge occurring at location $x$ in $l$ ).  We say that a loop sequence $s'$ is obtained from deformations of $s$ provided that exactly one component of $s'$ arises from deformation of one loop in $s$.  
	See  \cite[Fig.~6, Fig.~7]{Cha} and \cite[Fig.~4, Fig.~5]{Jafar} for graphical illustrations of positive and negative deformations.
	
	The sets $\mD^{+}(s)$ and $\mD^{-}(s)$ consist of all loop sequences obtained from positive or negative deformations of $s$, respectively. 
	\item	\label{expan} A positive expansion of $l$ at location $x$ by a plaquette  $p$ passing through $e^{-1}$ replaces $l$ with the pair of loops $(l,p)$.  A negative expansion of $l$ at location $x$ by a plaquette  $p$ passing through $e$ replaces $l$ with the pair of loops $(l,p)$.  See   \cite[Fig.~8, Fig.~9]{Jafar} for graphical illustrations of positive and negative expansions.
	The sets $\mE^{+}(s)$ and $\mE^{-}(s)$ consist of all loop sequences obtained from positive or negative expansions of $s$, respectively.
\end{enumerate}

In preparation for an application of the It\^o formula, we recall a convenient matrix analogue of It\^o differentials.  To treat each of the three groups $G$ in a unified way, we introduce parameters $\lambda$, $\nu$, and $\mu$ as follows.  For $G \in \{SO(N),U(N),SU(N)\}$, we rewrite \eqref{e:BB} as
\begin{equ} \label{eq:ItoDif}
	\dif B^{ij}\dif B^{k\ell} \eqdef \dif \< B^{ij} ,B^{k\ell}\> = \big ( \lambda \delta_{i\ell} \delta_{jk} \, \,\,+\nu \delta_{ij}\delta_{k\ell} \, \,\,+\mu \delta_{ik}\delta_{j \ell}\big)\ \dif t\;,
\end{equ}
where $\lambda ,\mu ,\nu$
depend on $G$ and the choice of the inner product. 
Since we will apply It\^o formula to  products of matrices, we will use a matrix variant of the standard $\dif B\dif B$ notation 
for formulating It\^o's rule.  Given a matrix $M$, we use the shorthand $\dif B M$ or $\dif B M \dif B$, which should always be understood by writing it in components as a matrix product and in the latter case applying \eqref{e:BB}.  This leads us to two useful identities which will be crucial in the proof of Theorem \ref{theo:main}. 

\medskip

Given adapted matrix-valued processes $M,N$, we have the following two identities
	\begin{align}
		\dif B M \dif B &=\big (  \lambda \tr M  \, \,\,+\nu M  \, \,\,+\mu M^{t} \big ) \dif t \label{eq:M}. \\
		\text{Tr}(\dif B M)\text{Tr}(\dif B N)&=\Big ( \lambda \text{Tr}(MN)+\nu\text{Tr}(M)\text{Tr}(N)+\mu\text{Tr}(MN^{t}) \Big ) \dif t \label{eq:MN}. 
	\end{align}
The first follows from \eqref{eq:ItoDif} by fixing components $i, \ell$ and writing
	\begin{align}
		(\dif B M \dif B)^{i\ell }=\sum_{j,k} \dif B^{ij}M^{jk} \dif B^{k\ell }
		&=\sum_{j,k}M^{jk} \big ( \lambda \delta_{i\ell} \delta_{jk}+ \nu \delta_{ij}\delta_{k\ell}+\mu \delta_{ik}\delta_{j \ell} \big)\ \dif t\;,  \nonumber \\
		&=\big ( \lambda \delta_{i \ell}\text{Tr}M +\nu M^{i \ell}+\mu M^{\ell i} \big ) \dif t \nonumber.
	\end{align}
The second follows in a similar way, since
	\begin{equs}
		\text{Tr}(\dif B M)\text{Tr}(\dif B N)
		&=\sum_{i,j,k,\ell}\dif B^{ij}M^{ji}\dif B^{k\ell}N^{\ell k} 
		=\sum_{i,j,k,\ell} M^{ji} N^{\ell k} 
		\big ( \lambda \delta_{i\ell} \delta_{jk} +\nu \delta_{ij}\delta_{k\ell} +\mu \delta_{ik}\delta_{j \ell}\big)\ \dif t  
		 \\
		&=\Big ( \lambda \text{Tr}(MN)+\nu\text{Tr}(M)\text{Tr}(N)+\mu\text{Tr}(MN^{t}) \Big ) \dif t \;.	
	\end{equs}
We remark that the identities \eqref{eq:M}-\eqref{eq:MN}
are not new, and are sometimes called the ``magic formulas'', see e.g. \cite[Lemma~7.1]{Dah2022II} for more background and literature, or \cite[Lemma~4.1]{MR2494192}.

We now turn to the proof of the main theorem.
\begin{proof}[Proof of Theorem \ref{theo:main}] 
	 Let $s$ be a string with minimal representation $(l_{1},\dots,l_{m})$, then we may view the quantity $\phi(s)$ as the mean of the stationary It\^o process $N^{-m} \Pi_{i=1}^mW_{l_i}$.  For each constituent loop $l \in \{ l_{1},\dots, l_{m} \}$, there exist edges (depending on $l$) labelled $e_{1},\dots, e_{n}\in E$ such that $W_{l}=\text{Tr}(Q_{e_{1}}\cdots Q_{e_{n}})$. Applying It\^o's product rule with \eqref{eq:YM} and \eqref{eq:YM*} yields
\begin{equs}[ME11]
\dif W_l
\;=\; \Big ( \frac{c_\mfg }{2} |l|W_{l}+ \mathcal{D}_{l}+\cI_{l}  \Big ) \dif t+\dif M_{l}, 
\end{equs}
where we define
\begin{align}
\mathcal{D}_{l} &\; \eqdef \; \sum_{x=1}^{n} 
\frac12 \Tr \Big( \prod_{i=1}^{x-1}  Q_{e_{i}}   \Big [ \nabla \cS(Q)_{e_{x}} \1_{e_x\in E^+}+ ( \nabla \cS(Q)_{e_{x}^{-1}})^* \1_{e_x^{-1}\in E^+}\Big ] \prod_{i=x+1}^{n}Q_{e_{i}} \Big ),  \label {eq:DDef} \\
\dif M_{l} & \; \eqdef \;  \sum_{x=1}^{n} 
 \Tr \Big( \prod_{i=1}^{x-1}  Q_{e_{i}}   \Big [ \dif B_{e_{x}}Q_{e_x}  \1_{e_x\in E^+}+Q_{e_x}\dif B_{e_{x}^{-1}}^* \1_{e_x^{-1}\in E^+}\Big ] \prod_{i=x+1}^{n}Q_{e_{i}} \Big ) , \nonumber\\
\cI_{l}\dif t &\; \eqdef \;
\sum_{x<y}
\, \Tr\big(Q_{a}\,\dif Q_{e_{x}}Q_{b}\,\dif Q_{e_{y}}Q_{c}\big),  \label{e:Idef}
\end{align}
and recall that $e_{x}\in E^{-}  \Leftrightarrow e_{x}^{-1} \in E^{+}$  (with the usual convention used above and below that an empty product of matrices 
 is $I_N$). 
In the definition of $\cI_{l}$ we use the shorthand notation 
$$
Q_{a}\eqdef \prod_{i=1}^{x-1}Q_{e_{i}} , \quad
Q_{b}\eqdef \prod_{i=x+1}^{y-1}Q_{e_{i}}, \quad
Q_{c}\eqdef \prod_{i=y+1}^{n}Q_{e_{i}}
$$
 and omit the dependence of these quantities on $x$ and $y$. 
 
Below, in Step \ref{MM1} we write $\cI_{l}$ in terms of the splitting and twisting operations, leading to the identity \eqref{eq:Il}.  At this stage we have the dynamic for each fixed loop in $s$ and now want to analyze $W_{s}$, which is itself a product of It\^o processes, so we apply It\^o's product rule again and obtain
\begin{align}\label{ME10}
\dif W_{s}=\dif  \big( \Pi_{i=1}^mW_{l_i} \big) 
&=\sum_{i=1}^m\dif  W_{l_i} \Pi_{j\neq i}W_{l_j} +\cI_{s}\dif t \nonumber \\
&= \frac{c_\mfg }{2} |s| W_{s}\dif t+\sum_{i=1}^{m} \big (  \cI_{l_{i}}+\mathcal{D}_{l_{i}} \big )\Pi_{j\neq i}W_{l_j} \dif t+\cI_{s} \dif t+\dif M_{s} , 
\end{align}
where $M_{s}$ is a martingale
and $\cI_{s}$ denotes the It\^o correction defined by
\begin{equation}
\cI_{s} \dif t \; \eqdef \; \sum_{i<j} 
	\dif W_{l_i} \dif W_{l_j} \Pi_{k\neq i,j} W_{l_k}, \nonumber
\end{equation}
which is calculated in Step \ref{MM2} in terms of the merger operation, leading to the identity \eqref{eq:Is}.  In Step \ref{MM3}, we express $\mathcal{D}_{l}$ in terms of the deformation and expansion operations using the expression \eqref{e:DS} for the drift $\nabla \cS$, leading to the identity \eqref{eq:d}. Next, we normalize $W_{s}$ by dividing \eqref{ME10} by $N^{m}$, taking expectation on both sides, and using stationarity to obtain
\begin{equation}
-\frac{c_\mfg }{2} |s| \phi(s)=\frac{1}{N^{m}}\E \Big [ \sum_{i=1}^{m} \big (  \cI_{l_{i}}+\mathcal{D}_{l_{i}} \big )\Pi_{j\neq i}W_{l_j} +\cI_{s} \Big ] .\label{e:FS}
\end{equation}
In the final step, we consider each particular Lie group, specifying $c_\mfg$ according to \eqref{e:c_g}, the parameters $\lambda, \mu, \nu$ as in  \eqref{e:BB}, and use the output of Steps \ref{MM1}-\ref{MM3} to show that the RHS of \eqref{e:FS} can be closed in terms of $\phi$, leading to the master equations \eqref{e:master}-\eqref{eq:phi1UN}.

	\newcounter{MM} 
	\refstepcounter{MM} 
	
	\medskip
	
{\sc Step} \arabic{MM}.\label{MM1}\refstepcounter{MM}
	In this step we analyze an individual loop $W_{l}$ and argue that 
\begin{equation}\label{eq:Il}
		\cI_{l}
	=-\frac\nu2(|l|-\ell(l))W_{l}
	-\frac\lambda2\sum_{s'\in \mS^-(l)}W_{s'}
	 +\frac\lambda2\sum_{s'\in \mS^+(l) }W_{s'}
	 +\frac\mu2\sum_{l'\in \mT^{-}(l)}W_{l'}
	 -\frac\mu2\sum_{l'\in \mT^{+}(l)}W_{l'}  
	\end{equation}
where $\ell(l)$ is as in \eqref{e:ell} (with $m=1$ there).
To prove the claim, we apply \eqref{eq:YM}+\eqref{eq:YM*}
to \eqref{e:Idef}. Since $(B_e)_{e\in E_\Lambda^+}$
 are independent, the contribution to \eqref{e:Idef} is restricted to $x,y$ with the property that $e_{y}=e_{x}$ or $e_{y}=e_{x}^{-1}$.  Since the dynamics in \eqref{ME11} also depends on the orientation of the edge, we sub-divide each case into two further cases $e_{x} \in E^{+}$ and $e_{x}^{-1}\in E^{+}$ to obtain the identity 
\begin{equ}[e:I4]
\cI_{l}\dif t
=\sum_{x<y}
\Big( J_{l}^{(1)} \1_{e_{y}=e_{x}\in E^+} 
 + J_{l}^{(2)} \1_{e_{y}^{-1}=e_{x}^{-1}\in E^+} 
+ J_{l}^{(3)}  \1_{e_{y}^{-1}=e_{x}\in E^+}  
+ J_{l}^{(4)} \1_{e_{y}=e_{x}^{-1}\in E^+}   \Big) ,
\end{equ}
where $ J_{l}^{(i)}, i=1, 2, 3, 4,$ are defined by
\footnote{When $e \in E^-$, \eqref{eq:YM*} yields
$\dif Q_e = \dif Q_{e^{-1}}^*  = (\cdots)+Q_{e^{-1}}^* \dif B_{e^{-1}}^*= (\cdots)+ Q_e \dif B_{e^{-1}}^*$ where $(\cdots)$ is the drift.}
\begin{equs}[e:J1234]
J_l^{(1)}\dif t & \eqdef
\Tr\Big(
Q_{a} \big( \dif B_{e_{x}}Q_{e_{x}} \big)  Q_{b} \big( \dif B_{e_{y}}Q_{e_{y}}\big) Q_{c}
\Big)\;,
\\
J_l^{(2)}\dif t &\eqdef 
 \Tr\Big(
 Q_{a} \big( Q_{e_{x}}\dif B_{e_{x}^{-1}}^* \big) Q_{b} 
 \big( Q_{e_{y}}\dif B_{e_{y}^{-1}}^* \big) Q_{c}\Big) \;,
 \\
J_l^{(3)} \dif t &\eqdef 
\Tr\Big(Q_{a}\big(\dif B_{e_{x}}Q_{e_{x}} \big) Q_{b} 
\big(Q_{e_{y}}\dif B_{e_{y}^{-1}}^{*} \big)Q_{c}\Big)  \;,
\\
J_l^{(4)} \dif t&\eqdef
\Tr\Big(Q_{a} \big(Q_{e_{x}}\dif B_{e_{x}^{-1}}^* \big)Q_{b} 
\big(\dif B_{e_{y}} Q_{e_{y}}\big) Q_{c}\Big) .
\end{equs}
Recall that $B_e^*=-B_e$. To analyze each term, we will apply \eqref{eq:M} with a suitable choice of $M$ and $B$, while taking into account the relation between $e_{x}$ and $e_{y}$ imposed by the indicator function according to \eqref{e:I4}.  
For $J_{l}^{(1)}$ and $J_{l}^{(2)}$, the role of $M$ is played by $Q_{e_{x}}Q_{b}$ and $Q_{b}Q_{e_{y}}$ respectively, while the role of $B$ is played by $B_{e_{x}}$ and $B_{e_{x}^{-1}}^{*}=-B_{e_{x}^{-1}}$ leading to 
\begin{equ}
J_l^{(1)} = J_l^{(2)}
=\lambda \text{Tr}\big(Q_{e_{x}}Q_{b}\big)\text{Tr}\big(Q_{a}Q_{e_{x}}Q_{c} \big)
+\nu \text{Tr}\big(Q_{a}Q_{e_{x}}Q_{b}Q_{e_{x}}Q_{c} \big)
 +\mu \text{Tr}\big(Q_{a} Q_{b^{-1}} Q_{c} \big)  \;,
\end{equ}
where we note that $\mu\neq 0$ only if $G=SO(N)$ in which case $Q^*=Q^t$. 
In a similar way, we obtain
\begin{equ}
J_l^{(3)}  = J_l^{(4)}
=-\lambda \text{Tr}\big( Q_{b} \big)\text{Tr}\big(Q_{a}Q_{c}\big)
	-\nu \text{Tr}\big(Q_{a}Q_{e_{x}} Q_{b} Q_{e_{x}^{-1}}Q_{c} \big)-\mu \text{Tr}\big(Q_{a}Q_{e_{x}} Q_{b^{-1}} Q_{e_{x}^{-1}}Q_{c} \big) ,
\end{equ}
 where we used cyclic invariance of the trace and $Q_{e_{x}}Q_{e_{y}}=Q_{e_{x}}Q_{e_{x}}^*=I_N$.

From the definition of splitting terms in \ref{split}, the terms above with a coefficient $\lambda$ contribute the splitting terms in \eqref{eq:Il}.  Indeed, recall that after choosing a closed path in $l$ and writing the loop as $l=e_{1}\cdots e_{n}$, the sets $\mS^{+}(l) $ and $\mS^{-}(l)$ consist of all loop sequences $s'=(\times_{x,y}^{1}l,\times_{x,y}^{2}l)$ where the locations $x \neq y$ have the property that $e_{y}=e_{x}$ or $e_{y}=e_{x}^{-1}$ respectively, according to whether the splitting is positive or negative.  Note that in the calculation above $x<y$, however if we let $s'':=(\times_{y,x}^{1}l,\times_{y,x}^{2}l)$, then $W_{s'}=W_{s''}$. 
Here, we recall from the discussion in \cite[Sec~2.2]{Cha} (below definitions of $\mS^-,\mS^+$ therein) that
if a loop $l$ has a splitting at $x$ and $y$, then it also has a splitting at $y$ and $x$, and they are reverse of each other, and should be counted as two distinct splittings of $l$. This gives the coefficient $\lambda/2$ before the splitting terms.
The same applies to  twisting and we will keep this in mind below.

The sets $\mT^{+}(l) $ and $\mT^{-}(l)$ consist of all loops of the form $\propto_{x,y} l$ where the locations $x \neq y$ have the property that $e_{y}=e_{x}$ or $e_{y}=e_{x}^{-1}$ respectively, according to whether the twisting is positive or negative.  Hence, from the definition of the twisting terms \ref{twist}, the terms  with a coefficient $\mu$
 lead to the final two terms in \eqref{eq:Il}.

Finally, we turn to the terms with coefficient $\nu$, and begin by introducing the following notation: for any  edge $e\in E^+$ we let $A(e)$ be the set of locations in $l$ where $e$ occurs and  $B(e)$ be the set of locations in $l$ where $e^{-1}$ occurs.
The terms with coefficient $\nu$ are given by 
\begin{equs}
	 &\nu W_{l}\sum_{x<y} \Big( \1_{e_{x}=e_{y}} -\1_{e_{x}=e_{y}^{-1}} \Big)
	\\&=  \nu W_l\sum_{x<y}
	\sum_{e\in E^+} \Big(\1_{e_x=e}\1_{e_y=e}+\1_{e_x=e^{-1}}\1_{e_y=e^{-1}}
	-  \1_{e_x=e}\1_{e_y=e^{-1}} - \1_{e_x=e^{-1}}\1_{e_y=e}\Big)
	\\
	&=\nu W_l\sum_{e\in E^+}\Big(\frac12|A(e)|(|A(e)|-1)+\frac12|B(e)|(|B(e)|-1)-|A(e)||B(e)|\Big)
\\&=-\frac \nu 2W_l\sum_{e\in E^+}\Big(|A(e)|+|B(e)|-(|A(e)|-|B(e)|)^2\Big)=-\frac{\nu}2 W_l(|l|-\ell(l)),
\end{equs}
with $\ell(l)=\sum_{e\in E^+}(|A(e)|-|B(e)|)^2$, which is precisely the first term in \eqref{eq:Il}, completing the proof of the claim. 
	
	\medskip

	{\sc Step} \arabic{MM}.\label{MM2}\refstepcounter{MM}
	In this step we consider the $m$ constituent Wilson loops $W_{l_{i}}$ in $W_{s}$ and argue that 
\begin{equation}
	\aligned
		\cI_{s}& = 
		\frac{\lambda}{2} \!\!\sum_{s'\in \mM^+_{U}(s) } \!\! W_{s'}
		-\frac{\mu}{2} \!\!\!\! \sum_{s'\in \mM^{+}(s) \setminus \mM^+_{U}(s) }\!\!\!\! W_{s'}
		-\frac{\lambda}{2}\!\! \sum_{s'\in \mM^-_{U}(s) }\!\! W_{s'}
		+\frac{\mu}{2} \!\!\!\! \sum_{s'\in \mM^{-}(s) \setminus \mM^-_{U}(s) }\!\!\!\! W_{s'} 
		\\&\qquad\qquad\qquad\qquad\qquad\qquad\qquad +\nu \sum_{i<j}\sum_{e\in E^+}t_i(e)t_j(e)W_s\;.
		\endaligned\label{eq:Is}
	   \end{equation}

To analyze $\cI_{s}$,  we start by fixing two loops $l_{i}$ and $l_{j}$ with $i \neq j$, and analyze $\dif W_{l_{i}}\dif W_{l_{j}}$.  First we choose a path to represent each loop and write them as
	$$l_i=\Pi_{k=1}^{|l_i|}e_k^i=ae_{x}^ib,\qquad l_j=\Pi_{k=1}^{|l_j|}e_k^j=ce_{y}^jd,$$
	where we use a shorthand notation
	$$a \eqdef \Pi_{k=1}^{x-1}e_k^i,\qquad b\eqdef\Pi_{k=x+1}^{|l_i|}e_k^i,\qquad c\eqdef\Pi_{k=1}^{y-1}e_k^j,\qquad d=\Pi_{k=y+1}^{|l_j|}e_k^j.$$
	Using again the independence of edges, taking into account \eqref{ME11} we obtain 	
	\begin{align}
	\dif W_{l_{i}}\dif W_{l_{j}} =\sum_{x=1}^{|l_{i}| }\sum_{y=1}^{|l_{j}| } (\1_{e_{y}^{j}=e_{x}^{i} }+\1_{e_{y}^j=(e_{x}^{i})^{-1}} )\tr(Q_a\dif Q_{e_x^{i}}Q_b)\tr(Q_c\dif Q_{e_y^{j} }Q_d). \nonumber
		\end{align}
	 To ease the notation below, we will drop the superscript from the edge and simply write $e_x=e_x^i$ and $e_y=e_y^j$. 
Apply \eqref{eq:YM}+\eqref{eq:YM*} to the r.h.s. of the above equation.
 Using $B_{e}^{*}=-B_{e}$ and cyclic invariance of the trace, we may  re-write the r.h.s. above as
		\begin{align}
		\sum_{x=1}^{|l_{i}| }\sum_{y=1}^{|l_{j}| } \Big ( \1_{e_{y}=e_{x}\in E^+}J_s^{(1)}+\1_{e_{y}^{-1}=e_{x}^{-1}\in E^+}J_s^{(2)}+\1_{e_{y}^{-1}=e_{x}\in E^+}J_s^{(3)}+\1_{e_{y}=e_{x}^{-1}\in E^+}J_s^{(4)} \Big )\label{ssz1}
	\end{align}
where $J_s^{(i)}$  are defined by
\begin{align*}
	 J_s^{(1)}\dif t &\eqdef \,
	 \tr\big(\dif B_{e_{x}} Q_{e_x}Q_bQ_a\big)
	 \tr(\dif B_{e_{y}}Q_{e_y}Q_dQ_c\big),
	 	\\
	J_s^{(2)}\dif t &\eqdef \,
	\tr\big(\dif B_{e_{x}^{-1}} Q_bQ_aQ_{e_x}\big)
	\tr\big(\dif B_{e_{y}^{-1} } Q_dQ_cQ_{e_y}\big),
	\\
	 J_s^{(3)}\dif t &\eqdef \,
	 -\tr\big(\dif B_{e_{x}} Q_{e_x}Q_bQ_a\big)
	 \tr\big(\dif B_{e_{y}^{-1} } Q_dQ_cQ_{e_y }\big),
	 \\
	 J_s^{(4)}\dif t \, &\eqdef \,
	 - \tr\big(\dif B_{e_{x}^{-1} } Q_bQ_aQ_{e_x}\big)
	 \tr\big(\dif B_{e_{y}}Q_{e_y}Q_dQ_c\big) .
	\end{align*}
We calculate these terms similarly as in Step~\ref{MM1}
using cyclic invariance of the trace, $Q_{e_{x}}Q_{e_{x}}^*=I_N$,
 the fact that $\mu\neq 0$ only if $G=SO(N)$ in which case $Q_{e_{x}}^*=Q_{e_{x}}^t$,
and taking into account the relation between $e_{x}$ and $e_{y}$ imposed by the indicator functions.

Applying \eqref{eq:MN} with $M=Q_{e_x}Q_bQ_a$ and $N=Q_{e_y}Q_dQ_c$
for $J_s^{(1)} $ and with $M=Q_bQ_aQ_{e_x}$ and $N=Q_dQ_cQ_{e_y}$ for $J_s^{(2)}$, under the assumption $e_x=e_y$ we obtain
	\begin{equ}
		J_s^{(1)} =J_s^{(2)} 
		=\lambda\tr\big( Q_a Q_{e_x}Q_dQ_cQ_{e_x}Q_b\big)
		+\nu\tr\big(Q_{l_i}\big)\tr\big(Q_{l_j}\big)
		+\mu\tr\big(Q_aQ_{c^{-1}}Q_{d^{-1}}Q_b\big).
	\end{equ}
Similarly, applying \eqref{eq:MN} with $M=Q_{e_x}Q_bQ_a$, $N=Q_dQ_cQ_{e_y}$ 
and $M=Q_bQ_aQ_{e_x}$, $N=Q_{e_y}Q_dQ_c$,
under the assumption $e_x=e_y^{-1}$ we obtain
\begin{equ}
	J_s^{(3)} = J_s^{(4)}
 =-\lambda\tr \big( Q_a Q_dQ_cQ_b\big)
-\nu\tr \big(Q_{l_i} \big)\tr \big(Q_{l_j} \big)
-\mu\tr\big(Q_aQ_{e_x}Q_{c^{-1}}Q_{d^{-1}}Q_{e_x}Q_b \big).
\end{equ}

We first note that for $s=(l_{1},\dots,l_{m} )$,  the set $\mM^+(s)$ is the collection of $s'$ which can be obtained from $s$ by merging some $l_{i}$ and $l_{j}$ with $i \neq j$.  If $l_{i}$ and $l_{j}$ can be merged at locations $x$ and $y$ respectively, then $s',s''$ defined by $s'=(l_{1},\dots, l_{i-1}, l_{i}\oplus_{x,y} l_j , l_{i},\dots,l_{m})$ and $s''=(l_{1},\dots, l_{j-1}, l_{i}\oplus_{y,x} l_j , l_{j},\dots,l_{m})$ both belong to $\mM^+(s)$.   Since $i<j$ we need to account for both contributions, leading to the factors of $1/2$ in \eqref{eq:Is}.  An analogous characterization holds for $\mM_{U}^+(s) , \mM^-(s),$ and $\mM_{U}^-(s)$.

To find the  contribution of  the terms with coefficient $\mu$
 to $\cI_{s}$, we recall the definition of mergers terms in \ref{merger} and multiply by $\prod_{k \neq i,j }W_{l_{j}}$ and sum over $i<j$ to obtain
\begin{align}
	&\mu \sum_{i<j}\sum_{x=1}^{|l_{i}| }\sum_{y=1}^{|l_{j}| } \Big (\1_{e_{y}=e_{x} } W_{l_{i}\ominus_{x,y} l_j }-\1_{e_{y}=e_{x}^{-1} } W_{l_{i}\oplus_{x,y} l_j } \Big )\prod_{k \neq i,j }W_{l_{j}} \nonumber \\
&=\frac{\mu}{2} \sum_{s'\in \mM^{-}(s) \setminus \mM^-_{U}(s) }W_{s'} -\frac{\mu}{2} \sum_{s'\in \mM^{+}(s) \setminus \mM^+_{U}(s) }W_{s'}\nonumber.
\end{align}
Similarly, collecting the terms with coefficient $\lambda$, multiplying by $\prod_{k \neq i,j }W_{l_{j}}$ and summing over $i<j$ yields
\begin{equation}
	\lambda \sum_{i<j}\sum_{x=1}^{|l_{i}| }\sum_{y=1}^{|l_{j}| } \Big (\1_{e_{y}=e_{x} } W_{l_{i}\oplus_{x,y} l_j }-\1_{e_{y}=e_{x}^{-1} } W_{l_{i}\ominus_{x,y} l_j } \Big )\prod_{k \neq i,j }W_{l_{j}}= \frac{\lambda}{2} \sum_{s'\in \mM^+_{U}(s) }W_{s'}-\frac{\lambda}{2} \sum_{s'\in \mM^-_{U}(s) }W_{s'}. \nonumber
\end{equation}		

Finally, we turn to the terms with coefficient $\nu$, and begin by introducing the following notation: for any  edge $e\in E^+$ we let $A_r(e)$ be the set of locations in $l_r$ where $e$ occurs and  $B_r(e)$ be the set of locations in $l_r$ where $e^{-1}$ occurs.
The terms with coefficient $\nu$ are given by $\sum_{x,y}$ of
\begin{equs}
	& \nu W_{l_i}W_{l_j}   \Big(  \1_{e_{x}=e_{y}}-\1_{e_{x}=e_{y}^{-1}} \Big)
	\\&= 
	\nu W_{l_i}W_{l_j}\sum_{e\in E^+} \Big(\1_{e_x=e}\1_{e_y=e}+\1_{e_x=e^{-1}}\1_{e_y=e^{-1}}
	-  \1_{e_x=e}\1_{e_y=e^{-1}}-\1_{e_x=e^{-1}}\1_{e_y=e}\Big).
\end{equs}
Summing over $x,y$ we obtain
\begin{align*}
	&\nu W_{l_i}W_{l_j}\sum_{e\in E^+}\Big(|A_i(e)||A_j(e)|+|B_i(e)||B_j(e)|-|A_i(e)||B_j(e)|-|B_i(e)||A_j(e)|\Big)
	\\&= \nu W_{l_i}W_{l_j} \sum_{e\in E^+}\Big((|A_i(e)|-|B_i(e)|)(|A_j(e)|-|B_j(e)|)\Big)=\nu W_{l_i}W_{l_j}\sum_{e\in E^+} t_i(e)t_j(e),
\end{align*}
with $t_i(e)=|A_i(e)|-|B_i(e)|$.
 Substituting the $\nu$ part into $\cI_s$, we get the last term in \eqref{eq:Is}.

	\medskip
	
	{\sc Step} \arabic{MM}.\label{MM3}\refstepcounter{MM}
	In this step, we analyze the gradient terms $\mathcal{D}_{l}$ and claim that 
	\begin{equation}
	\mathcal{D}_{l}=
	\begin{cases}
	-\frac14N\beta\sum_{s'  \in \mD^+(l) }W_{s'}\quad \quad \, \,\,  +\frac14N\beta\sum_{s' \in \mD^-(l) }W_{s'}  &\quad \text{for} \quad G \in \{SO(N), U(N) \}\label{eq:d} \\  
	-\frac14N\beta\sum_{s' \in \mD^+(l) \cup \mE^+(l) }W_{s'}+\frac14N\beta\sum_{s' \in \mD^-(l ) \cup \mE^-(l) }W_{s'} & \quad \text{for} \quad G=SU(N) 
	\end{cases}
	\end{equation}	
	Recalling $\frac12\nabla\cS(Q)_e$ given in \eqref{e:DS},
we first claim that for each of our Lie groups, and for every $e\in E^-$,
\begin{equ}[e:DSDS]
\frac12 (\nabla S(Q))_{e^{-1}}^* = \frac12	\nabla S(Q)_{e}
\end{equ}
where the r.h.s. is given by the formula  \eqref{e:DS}.
Indeed, to calculate the l.h.s. of \eqref{e:DSDS}, note that
$$
\sum_{p\in \cP_{\Lambda},p\succ e^{-1}}
\!\!\! \Big((Q_p-Q_p^*)Q_{e^{-1}}\Big)^*
=
\sum_{p\in \cP_{\Lambda},p\succ e^{-1}}
\!\!\! Q_e (Q_p^* -Q_p)
=
\sum_{\bar p\in \cP_{\Lambda}, \bar p\succ e}
\! (Q_{\bar p}-Q_{\bar p}^*)Q_{e}
$$
where in the last step
we made a change of variable
$p = e^{-1}e_2e_3e_4
\mapsto 
\bar p = e\, e_4^{-1}e_3^{-1}e_2^{-1}$.  This establishes \eqref{e:DSDS} in case where $G \in \{ SO(N),U(N)\}$.  For the $SU(N)$ case we need to analyze the additional trace term, so noting that $\tr(Q_p-{Q}_p^{*}) $ is purely imaginary,
\begin{equs}
\sum_{p\in \cP_{\Lambda},p\succ e^{-1}} \!\!\!
\Big(\tr(Q_p-{Q}_p^{*}) Q_{e^{-1}}\Big)^*
&=
-  \!\!\! \!\!\! \sum_{p\in \cP_{\Lambda},p\succ e^{-1}} \!\!\!
\tr(Q_p-{Q}_p^{*}) Q_{e}
\\
&=
-  \!\!\! \! \sum_{\bar p\in \cP_{\Lambda}, \bar p\succ e} 
\tr(Q_{\bar p}^{-1}  - Q_{\bar p}) Q_{e}
=\!\!
\sum_{\bar p\in \cP_{\Lambda}, \bar p\succ e}\!\!
\tr(Q_{\bar p} - Q_{\bar p}^*) Q_{e}
\end{equs}
so \eqref{e:DSDS} holds.  In light of \eqref{e:DSDS}, the constraint in \eqref{eq:DDef} on the orientation of the edge may be removed, and the expression for $\mathcal{D}_{l}$ simplifies to
\begin{equ}
\mathcal{D}_{l}=\frac{1}{2}\sum_{x=1}^{n} 
\Tr \Big( \prod_{i=1}^{x-1}  Q_{e_{i}} 
  \nabla \cS(Q)_{e_{x}} 
 \prod_{i=x+1}^{n}Q_{e_{i}} \Big ). 
\end{equ}
In light of \eqref{e:DS}, in the case $G \in \{SO(N),U(N)\}$, we find that $\mathcal{D}_{l}$ is given by
\begin{equ}
-\frac14N\beta 
\sum_{x=1}^{n} \sum_{p\in \cP_{\Lambda}, p\succ e_x}
\Tr \Big( \prod_{i=1}^{x-1}  Q_{e_{i}} 
	\Big(Q_p-Q_{p}^{-1}\Big)Q_{e_x}  
 \prod_{i=x+1}^{n}Q_{e_{i}} \Big ) 
=
-\frac14N\beta \sum_{x=1}^n\sum_{p\succ e_x}
\Big(W_{l\oplus_xp}-W_{l\ominus_xp}\Big).
\end{equ}
which yields the first case in \eqref{eq:d} taking into account that each $s' \in \mD^+(l)$ is of the form $l\oplus_xp$ for a plaquette $p$ which contains an edge $e_{x}$ and analogously for $s' \in \mD^-(l)$.  Note that here and below, in the cycle $p$ we choose the path where $e_{x}$ is the first edge.

For the $SU(N)$ case, the quantity $\mathcal{D}_{l}$ is given by adding to the quantity above the term
\begin{align}\label{eq:sus}
\frac{\beta}4
\sum_{x=1}^{n} \sum_{p\in \cP_{\Lambda}, p\succ e_x}
\Tr \Big( 
 \prod_{i=1}^{n}Q_{e_{i}} \Big ) 
 \tr\Big(Q_p-Q_{p}^{-1}\Big)
 =
 \frac{\beta}4
\sum_{x=1}^{n} \sum_{p\in \cP_{\Lambda}, p\succ e_x}
 \Big( W_l W_p - W_l W_{p^{-1}}
\Big)\;,
\end{align}
which yields the second case in \eqref{eq:d} taking into account that each $s' \in \mE^-(l)$ is of the form $s'=(l,p)$ for a plaquette $p$ which contains an edge $e_{x}$ and analogously for $s' \in \mE^+(l)$. Here we use $W_{p^{-1}}=W_{\tilde  p}$ with $\tilde p\succ e_x^{-1}$.

{\sc Step} \arabic{MM}.\label{MM4}\refstepcounter{MM}
In this final step, we consider each group $G \in \{SO(N),U(N),SU(N)\}$ and use Steps \ref{MM1}-\ref{MM3} to conclude the proof of the master equations \eqref{e:master}-\eqref{eq:phi1UN}.  To this end, we use the identity \eqref{eq:Il} for each constituent loop $l_{i}$ and insert it, together with \eqref{eq:Is} and \eqref{eq:d} into \eqref{e:FS}.  We note in advance that  
\begin{equation}
	N^{-m}\sum_{i=1}^{m}\sum_{s'\in \mathcal{O}(l_{i} ) }W_{s'}\prod_{j \neq i }W_{l_{j}}=N^{-m}\sum_{s' \in \mathcal{O}(s) } W_{s'}  \label{eq:s4}
\end{equation}
for each of the operations $\mathcal{O}\in \{\mS^\pm,\mT^\pm,\mD^\pm, \mE^\pm \}$.  Depending on the operation $\mathcal{O}$, the quantity $W_{s'}$ on the r.h.s. of \eqref{eq:s4} should be normalized in one of three possible ways according to the following:
\minilab{eq:ss}
\begin{equs}
s' \in \{ \mM^+(s), \mM^-(s)\} \quad & \rightarrow \quad \frac{\E W_{s'}}{N^{m-1} }  =\phi(s'). \label{eq:s3} 
\\
s' \in \{\mD^+(s), \mD^-(s), \mT^+(s), \mT^-(s)\} \quad  &\rightarrow \, \, \quad  \frac{\E W_{s'}}{N^{m} }  =\phi(s'). \label{eq:s1} 
\\
s' \in \{\mS^+(s), \mS^-(s), \mE^+(s), \mE^-(s)\} \quad  &\rightarrow \,  \quad  \frac{\E W_{s'}}{N^{m+1} } =\phi(s'). \label{eq:s2}
\end{equs}  
This follows since we have $m-1$ loops in $W_{s'}$ for $s' \in \{\mM^+(s), \mM^-(s)\}$, $m$ loops in $W_{s'}$ for $s'\in\{\mD^+(s),\mD^-(s),, \mT^+(s), \mT^-(s)\}$, and $m+1$ loops in $W_{s'}$ for $s'\in \{\mS^+(s), \mS^-(s)\}$.  Also we have an extra $W_p$ for the expansion term, which requires an extra $\frac1N$ in $W_{s'}$ for $s'\in \{\mE^+(s), \mE^-(s)\} $.   

\medskip

We now turn to each of the groups and simplify the r.h.s. of \eqref{e:FS} in accordance with the observations above.

Let $G=SO(N)$. 
By \eqref{e:BB1},  we have $\lambda=-\frac12$, $\mu=\frac12$, $\nu=0$. Since $c_{\so(N)}=-\frac12(N-1)$ from \eqref{e:c_g}, 
the identity \eqref{e:FS} takes the form
\begin{equation}
(N-1)|s|\phi(s)
=\frac{4}{N^{m}}\E \Big [ \sum_{i=1}^{m} ( \mathcal{D}_{l_{i}}  
 + \cI_{l_{i}}) \Pi_{j\neq i}W_{l_j} +\cI_{s} \Big ] \nonumber.
\end{equation}
Substituting \eqref{eq:d} for $\mathcal{D}_{l_{i}}$,
\eqref{eq:Il} for $\cI_{l_{i}}$ and  \eqref{eq:Is}  for $\cI_{s}$,
 the r.h.s. is equal to
\begin{equs}
{} & - N\beta \sum_{s'\in \mathbb{D}^+(s)} \frac{\E[  W_{s'}]}{N^{m}}
 + N\beta \sum_{s'\in \mathbb{D}^-(s)}\frac{\E[ W_{s'}]}{N^{m}} 
 -N \sum_{s'\in \mS^+(s)}\frac{ \E[W_{s'}] }{N^{m+1}}
 + N \sum_{s'\in \mS^-(s)}\frac{\E[W_{s'}] }{N^{m+1}}
 \\
 &\qquad\quad - \sum_{s'\in \mT^+(s)}\frac{\E[W_{s'}] }{N^{m}}
 + \sum_{s'\in \mT^-(s)}\frac{\E[W_{s'}] }{N^{m}} 
-\frac1{N}\sum_{s'\in \mM^+(s)}\frac{ \E[W_{s'}] }{N^{m-1}}
+\frac1{N}\sum_{s'\in \mM^-(s)}\frac{ \E[W_{s'}] }{N^{m-1}} 
\end{equs}
where we used \eqref{eq:s4}.      
Identifying the summands with $\phi$ by \eqref{eq:ss} completes the proof of \eqref{e:master}. 

Let $G=U(N)$. By \eqref{e:BB2},  we have $\lambda=-1$, $\mu=\nu=0$.
Since $c_{\u(N)}=-N$, the identity \eqref{e:FS} takes the form
\begin{equ}
N|s|\phi(s)
=\frac{2}{N^{m}}\E \Big [ \sum_{i=1}^{m} ( \mathcal{D}_{l_{i}}  
 + \cI_{l_{i}}) \Pi_{j\neq i}W_{l_j} +\cI_{s} \Big ] .
\end{equ}
Again substituting \eqref{eq:d}+\eqref{eq:Il}+\eqref{eq:Is} and using \eqref{eq:s4}, the r.h.s.
is equal to
\begin{equs}[e:UN]
{}& -\frac{N\beta}{2}\sum_{s'\in \mathbb{D}^+(s)}\frac{ \E[W_{s'}] }{N^{m}}
+\frac{N\beta}{2} \sum_{s'\in \mathbb{D}^-(s)}\frac{ \E[W_{s'}] }{N^{m}}	
\\
&\quad -\frac1{N}\sum_{s'\in \mM^+_U(s)}\frac{\E[W_{s'}] }{N^{m-1}}
+\frac1{N}\sum_{s'\in \mM^-_U(s)}\frac{ \E[W_{s'}] }{N^{m-1}}
-N \sum_{s'\in \mS^+(s)}\frac{\E[W_{s'}] }{N^{m+1}}
+N\sum_{s'\in \mS^-(s)}\frac{\E[W_{s'}] }{N^{m+1}} \;.
\end{equs}
Identifying the summands with $\phi$ by \eqref{eq:ss} completes the proof of \eqref{eq:phi1UN}.

Let $G=SU(N)$.
By \eqref{e:BB3}, we have $\lambda=-1$, $\mu=0$, $\nu=\frac{1}{N}$. 
Since $c_{\su(N)}=-N+\frac{1}{N}$ by \eqref{e:c_g}, 
the identity \eqref{e:FS} takes the form
\begin{equ}
(N-\frac{1}{N})|s|\phi(s)
=\frac{2}{N^{m}}\E \Big [ \sum_{i=1}^{m}  (\mathcal{D}_{l_{i}}  + \cI_{l_{i}} )\Pi_{j\neq i}W_{l_j} +\cI_{s} \Big ] . 
\end{equ}
We again  apply \eqref{eq:d}+\eqref{eq:Il}+\eqref{eq:Is} using \eqref{eq:s4}, and note that 
the only differences from the $U(N)$ case are the $\nu$-terms and the expansion terms.
The r.h.s. is then equal to 
\begin{equs}[e:SUN]
\eqref{e:UN} &\;  -\frac{N\beta}{2}\sum_{s'\in \mathbb{E}^+(s)}\frac{\E[W_{s'}] }{N^{m+1}}+\frac{N\beta}{2}\sum_{s'\in \mathbb{E}^-(s)}\frac{\E[W_{s'}] }{N^{m+1}}
\\
&-\frac{1}{N}|s|\frac{\E[W_{s}] }{N^{m}}+\frac{1}{N}\sum_{e\in E^+}\Big(\sum_{i=1}^mt_i(e)^2+2\sum_{i<j}t_i(e)t_j(e)\Big)\frac{\E[W_{s}] }{N^{m}}
\end{equs}
where we used $\sum_{i=1}^m |l_i|=|s|$ and $\ell(l_i)=\sum_{e\in E^+}t_i(e)^2$.
By \eqref{e:ell},
the second line of \eqref{e:SUN} is equal to $-\frac{1}{N}(|s|-\ell(s))\phi(s)$.
This implies \eqref{e:jafar} and completes the proof.
\end{proof}

\bibliographystyle{alphaabbr}
\bibliography{refs}

\newcommand{\etalchar}[1]{$^{#1}$}
\def\cprime{$'$} \def\polhk#1{\setbox0=\hbox{#1}{\ooalign{\hidewidth
  \lower1.5ex\hbox{`}\hidewidth\crcr\unhbox0}}}
\begin{thebibliography}{CCHS22b}

\bibitem[AK17]{AndersonKruczenski}
P.~D. Anderson and M.~Kruczenski.
\newblock Loop equations and bootstrap methods in the lattice.
\newblock {\em Nuclear Physics B}, 921:702--726, 2017.

\bibitem[BG18]{MR3861073}
R.~Basu and S.~Ganguly.
\newblock {${\rm SO}(N)$} lattice gauge theory, planar and beyond.
\newblock {\em Comm. Pure Appl. Math.}, 71(10):2016--2064, 2018.

\bibitem[BKK{\etalchar{+}}85]{batrouni1985}
G.~Batrouni, G.~Katz, A.~S. Kronfeld, G.~Lepage, B.~Svetitsky, and K.~Wilson.
\newblock Langevin simulations of lattice field theories.
\newblock {\em Physical Review D}, 32(10):2736, 1985.

\bibitem[CCHS22a]{CCHS2d}
A.~{Chandra}, I.~{Chevyrev}, M.~{Hairer}, and H.~{Shen}.
\newblock {Langevin dynamic for the {2D Yang--Mills} measure}.
\newblock {\em Publ. Math. IHÉS}, 2022.

\bibitem[CCHS22b]{CCHS3d}
A.~Chandra, I.~Chevyrev, M.~Hairer, and H.~Shen.
\newblock Stochastic quantisation of {Yang-Mills-Higgs in 3D}.
\newblock {\em arXiv preprint arXiv:2201.03487}, 2022.

\bibitem[CGMS09]{CGM2009}
B.~Collins, A.~Guionnet, and E.~Maurel-Segala.
\newblock Asymptotics of unitary and orthogonal matrix integrals.
\newblock {\em Advances in Mathematics}, 222(1):172--215, 2009.

\bibitem[Cha19a]{Cha}
S.~Chatterjee.
\newblock Rigorous solution of strongly coupled {$SO(N)$} lattice gauge theory
  in the large {$N$} limit.
\newblock {\em Comm. Math. Phys.}, 366(1):203--268, 2019.

\bibitem[Cha19b]{Chatterjee18}
S.~Chatterjee.
\newblock Yang-{M}ills for probabilists.
\newblock In {\em Probability and analysis in interacting physical systems},
  volume 283 of {\em Springer Proc. Math. Stat.}, pages 1--16. Springer, Cham,
  2019.

\bibitem[CS23]{Chevyrev2023}
I.~Chevyrev and H.~Shen.
\newblock Invariant measure and universality of the {2D Yang--Mills Langevin}
  dynamic.
\newblock {\em arXiv preprint arXiv:2302.12160}, 2023.

\bibitem[Dah16]{MR3554890}
A.~Dahlqvist.
\newblock Free energies and fluctuations for the unitary {B}rownian motion.
\newblock {\em Comm. Math. Phys.}, 348(2):395--444, 2016.

\bibitem[DGHK17]{MR3631396}
B.~K. Driver, F.~Gabriel, B.~C. Hall, and T.~Kemp.
\newblock The {M}akeenko-{M}igdal equation for {Y}ang-{M}ills theory on compact
  surfaces.
\newblock {\em Comm. Math. Phys.}, 352(3):967--978, 2017.

\bibitem[DHK17]{MR3613519}
B.~K. Driver, B.~C. Hall, and T.~Kemp.
\newblock Three proofs of the {M}akeenko-{M}igdal equation for {Y}ang-{M}ills
  theory on the plane.
\newblock {\em Comm. Math. Phys.}, 351(2):741--774, 2017.

\bibitem[DL22]{Dah2022II}
A.~Dahlqvist and T.~Lemoine.
\newblock Large {N} limit of the {Yang-Mills} measure on compact surfaces {II:
  Makeenko-Migdal} equations and planar master field.
\newblock {\em arXiv preprint arXiv:2201.05886}, 2022.

\bibitem[Dri19]{MR3982691}
B.~K. Driver.
\newblock A functional integral approaches to the {M}akeenko-{M}igdal
  equations.
\newblock {\em Comm. Math. Phys.}, 370(1):49--116, 2019.

\bibitem[Egu79]{Eguchi1979}
T.~Eguchi.
\newblock Strings in {$U(N)$} lattice gauge theory.
\newblock {\em Physics Letters B}, 87(1-2):91--96, 1979.

\bibitem[EK82]{EguchiKawai1982}
T.~Eguchi and H.~Kawai.
\newblock Reduction of dynamical degrees of freedom in the large-n gauge
  theory.
\newblock {\em Physical Review Letters}, 48(16):1063, 1982.

\bibitem[Foe79]{Foerster1979}
D.~Foerster.
\newblock {Yang--Mills} theory-a string theory in disguise.
\newblock {\em Physics Letters B}, 87(1-2):87--90, 1979.

\bibitem[GAO83]{Gonzalez1983}
A.~Gonzalez-Arroyo and M.~Okawa.
\newblock Twisted--{E}guchi--{K}awai model: A reduced model for large-{N}
  lattice gauge theory.
\newblock {\em Physical Review D}, 27(10):2397, 1983.

\bibitem[GAO14]{Gonzalez2014}
A.~Gonz{\'a}lez-Arroyo and M.~Okawa.
\newblock Testing volume independence of {$SU(N)$} pure gauge theories at large
  {$N$}.
\newblock {\em Journal of High Energy Physics}, 2014(12):1--33, 2014.

\bibitem[GL83]{guha1983}
A.~Guha and S.-C. Lee.
\newblock Stochastic quantization of matrix and lattice gauge models.
\newblock {\em Physical Review D}, 27(10):2412, 1983.

\bibitem[GN15]{GuionnetNovak}
A.~Guionnet and J.~Novak.
\newblock Asymptotics of unitary multimatrix models: the {Schwinger--Dyson}
  lattice and topological recursion.
\newblock {\em Journal of Functional Analysis}, 268(10):2851--2905, 2015.

\bibitem[Hsu02]{Hsu}
E.~P. Hsu.
\newblock {\em Stochastic analysis on manifolds}, volume~38 of {\em Graduate
  Studies in Mathematics}.
\newblock American Mathematical Society, Providence, RI, 2002.

\bibitem[Jaf16]{Jafar}
J.~Jafarov.
\newblock Wilson loop expectations in {$ SU (N) $} lattice gauge theory.
\newblock {\em arXiv preprint arXiv:1610.03821}, 2016.

\bibitem[L{\'e}v17]{Levy11}
T.~L{\'e}vy.
\newblock The master field on the plane.
\newblock {\em Ast\'{e}risque}, (388):ix+201, 2017.

\bibitem[LR15]{MR3410409}
W.~Liu and M.~R\"{o}ckner.
\newblock {\em Stochastic partial differential equations: an introduction}.
\newblock Universitext. Springer, Cham, 2015.

\bibitem[Mak02]{MakeenkoBook}
Y.~Makeenko.
\newblock {\em Methods of contemporary gauge theory}.
\newblock Cambridge University Press, 2002.

\bibitem[MM79]{MM1979}
Y.~M. Makeenko and A.~A. Migdal.
\newblock Exact equation for the loop average in multicolor {QCD}.
\newblock {\em Physics Letters B}, 88(1-2):135--137, 1979.

\bibitem[Sen08]{MR2494192}
A.~N. Sengupta.
\newblock Traces in two-dimensional {QCD}: the large-{$N$} limit.
\newblock In {\em Traces in number theory, geometry and quantum fields},
  Aspects Math., E38, pages 193--212. Friedr. Vieweg, Wiesbaden, 2008.

\bibitem[SZZ23]{SZZ22}
H.~Shen, R.~Zhu, and X.~Zhu.
\newblock A stochastic analysis approach to lattice {Y}ang-{M}ills at strong
  coupling.
\newblock {\em Comm. Math. Phys.}, 400(2):805--851, 2023.

\end{thebibliography}

\end{document}